\documentclass[12pt]{amsart}

\usepackage{comment}
\usepackage{appendix}
\usepackage{graphicx}
\usepackage{amsfonts}
\usepackage{amsmath}
\allowdisplaybreaks
\usepackage{amsthm}
\usepackage{xcolor}
\usepackage{mathtools}
\usepackage{fancyhdr,amssymb}
\usepackage{url}
\usepackage{enumerate}
\usepackage{booktabs}

\newcommand{\Z}{\mathbb{Z}}

\def\<{\langle}
\def\>{\rangle}

\newcommand{\C}{\mathbb{C}}
\newcommand{\supp}{\mathrm{supp}}
\newcommand{\Conj}{\mathrm{Conj}}

\newtheorem{theorem}{Theorem}[section]
\newtheorem{proposition}[theorem]{Proposition}

\newtheorem{definition}[theorem]{Definition}
\newtheorem{notation}[theorem]{Notation}
\newtheorem{remark}[theorem]{Remark}

\newtheorem{conjecture}[theorem]{Conjecture}

\newtheorem{lemma}[theorem]{Lemma}
\newtheorem{question}[theorem]{Question}

\begin{document}

\pagestyle{fancy}
\fancyhf{}
\lhead{}
\rhead{}
\cfoot{}
\renewcommand{\headrulewidth}{0pt}
\renewcommand{\footrulewidth}{0pt}


\title[A counterexample to the Gowers uncertainty conjecture]{A counterexample to the Gowers $U^k$ uncertainty conjecture for $k\ge 6$}

\author{Iana Vranesko}\email{iana\_vranesko@brown.edu}\address{Department of Mathematics, Brown University, Providence, RI, 02912}

\begin{abstract}
Let $f:\Z_N^d\to\C$ be a nonzero signal with support $E=\supp(f)$ and Fourier support $\Sigma=\supp(\hat f)$. The Donoho--Stark uncertainty principle asserts $N^d\le|E|\,|\Sigma|$, and the additive energy uncertainty principle of Aldahleh--Iosevich--Iosevich--Jaimangal--Mayeli--Pack refines it to $1\le|E|\cdot\|1_\Sigma\|_{U^2}^{4/3}$. It has been conjectured that the natural higher-order analogue
\[
1\le |E|\cdot\|1_\Sigma\|_{U^k}^{2^k/(k+1)}
\]
holds for every $k\ge 2$, the coset case giving equality. We show that this fails for every $k\ge 6$. The counterexample is the two-point signal $f=\delta_0-\delta_1$ on $\Z_3$, for which the left-hand quantity equals $\tfrac23\bigl(2(k+1)\bigr)^{1/(k+1)}$; this is $<1$ precisely when $k\ge 6$, and tends to $2/3$. More generally, for every prime $p\ge 3$ the analogous two-point signal on $\Z_p$ violates the conjecture for all sufficiently large $k$, the quantity tending to $2/p$. This happens because the $U^k$ norm of an indicator is, in the $k$ limit, governed not by the size of $\Sigma$ but by the size of the largest coset contained in $\Sigma$, so that in the limit the conjecture degenerates into the false assertion that $|E|$ times the largest coset inside $\Sigma$ is at least $N^d$. The conjectured inequality remains a theorem for $k=2$. The cases $k=3,4,5$, where the exponent $2^k/(k+1)$ is not yet large enough for these families to cause problems, are left open, together with the problem of determining the sharp exponent.
\end{abstract}

\maketitle

\section{Introduction}

Let $N\ge2$ and let $f:\Z_N^d\to\C$, with the discrete Fourier transform normalized by
\[
\hat f(m)\coloneq N^{-d/2}\sum_{x\in\Z_N^d}f(x)\chi(-m\cdot x),
\qquad
\chi(t)\coloneq e^{2\pi i t/N}.
\]
We write $E=\supp(f)$ and $\Sigma=\supp(\hat f)$. The classical uncertainty principle of Donoho and Stark \cite{DonohoStark} states that
\begin{equation}\label{eq:DS}
N^d\le |E|\cdot|\Sigma|
\end{equation}
for every nonzero $f$, with equality exactly when $f$ is a modulated multiple of the indicator function of a coset of a subgroup of $\Z_N^d$ (see \cite{MatolcsiSzucs}).

Aldahleh--Iosevich--Iosevich--Jaimangal--Mayeli--Pack \cite{2025additiveenergyuncertaintyprinciple} refined \eqref{eq:DS} by replacing $|\Sigma|$ with the additive energy
$\Lambda_2(\Sigma)=\#\{(x_1,x_2,x_3,x_4)\in\Sigma^4: x_1+x_2=x_3+x_4\}$, obtaining
\begin{equation}\label{eq:additive}
N^d\le |E|\,\Lambda_2(\Sigma)^{1/3}.
\end{equation}

We now recall the Gowers uniformity norms, in the normalization of \cite{2025additiveenergyuncertaintyprinciple} (see also \cite[Chapter 11]{Tao-Vu-2006}).

\begin{definition}[Gowers $U^k$ norm]\label{def:gowers}
Let $k\ge 2$ and $f:\Z_N^d\to\C$. Then $\|f\|_{U^k}$ is the unique nonnegative real $2^k$-th root of
\[
\|f\|_{U^k}^{2^k}
=
\frac{1}{N^{d(k+1)}}
\sum_{x\in\Z_N^d}\ \sum_{\substack{h_1,\dots,h_k\in\Z_N^d\\ h=(h_1,\dots,h_k)}}\ \prod_{w\in\{0,1\}^k}\Conj^{|w|}f(x+w\cdot h),
\]
where $\Conj^\ell(z)$ denotes $\ell$-fold conjugation and $w\cdot h=\sum_{i=1}^k w_ih_i$.
\end{definition}

Since $\Lambda_2(S)=N^{3d}\|1_S\|_{U^2}^4$, the inequality \eqref{eq:additive} may be rewritten as
\begin{equation}\label{eq:U2}
1\le |E|\cdot\|1_\Sigma\|_{U^2}^{4/3},
\end{equation}
which is strictly stronger than \eqref{eq:DS} unless $\Sigma$ is a coset of a subgroup.

The shape of \eqref{eq:U2} suggests the following extrapolation, which appears as a conjecture in the ``future work'' section of \cite{RefinedUncertainty}.

\begin{conjecture}[{\cite[Conjecture 5.2]{RefinedUncertainty}}]\label{conj:main}
Let $f:\Z_N^d\to\C$ be a nonzero signal with support $E$ and Fourier support $\Sigma$. Then for every $k\ge2$,
\begin{equation}\label{eq:conj}
1\le |E|\cdot\|1_\Sigma\|_{U^k}^{2^k/(k+1)}.
\end{equation}
\end{conjecture}

The evidence for Conjecture \ref{conj:main} is that it is a theorem for $k=2$, by \eqref{eq:additive}, and that it holds with equality for every $k$ in the extremal case of \eqref{eq:DS}: if $H\le\Z_N^d$ and $f=1_{y+H}$, then $\Sigma=H^\perp$, and a direct computation from Definition \ref{def:gowers} gives $\|1_{H^\perp}\|_{U^k}^{2^k/(k+1)}=|H^\perp|/N^d$, so that
\[
|\supp(1_{y+H})|\cdot\|1_{H^\perp}\|_{U^k}^{2^k/(k+1)}=|H|\cdot\frac{|H^\perp|}{N^d}=1.
\]
Moreover, the exponent $2^k/(k+1)$ is the largest one for which the coset examples do not already fail; see Remark \ref{rem:exponent}.

The purpose of this note is to show that Conjecture \ref{conj:main} is nevertheless false, and to identify where it breaks.

\begin{theorem}\label{thm:main}
Let $N=3$, $d=1$, and let $f=\delta_0-\delta_1:\Z_3\to\C$, so that $E=\{0,1\}$ and $\Sigma=\{1,2\}$. Then for every $k\ge 2$,
\begin{equation}\label{eq:mainvalue}
|E|\cdot\|1_\Sigma\|_{U^k}^{2^k/(k+1)}
=
\frac{2}{3}\bigl(2(k+1)\bigr)^{1/(k+1)} .
\end{equation}
This quantity is strictly less than $1$ if and only if $k\ge 6$, and it decreases to $2/3$ as $k\to\infty$. In particular Conjecture \ref{conj:main} fails for every $k\ge6$.
\end{theorem}

The example is relatively small: a two-point signal on a group of order three. The same signal is available at every $k$, and the point is only that the quantity \eqref{eq:mainvalue} decreases in $k$ and does not cross $1$ until $k=6$. At $k=3$ it equals $\tfrac23\cdot 8^{1/4}=1.1212\ldots$, which is far above $1$. This is why the failure is hard to spot at the first few orders of the Gowers hierarchy, which are the ones the conjecture was extrapolated from. This also explains some of the difficulty the author and the coauthors of \cite{RefinedUncertainty} had in trying to prove or disprove the conjecture: those efforts were largely focused on $k=3.$

Theorem \ref{thm:main} can be re-states for other $\mathbb{Z}_p$, where $p$ is prime.

\begin{theorem}\label{thm:allp}
Let $p\ge3$ be prime and let $f=\delta_0-\delta_1:\Z_p\to\C$, so that $E=\{0,1\}$ and $\Sigma=\Z_p\setminus\{0\}$. Then
\[
\lim_{k\to\infty}\ |E|\cdot\|1_\Sigma\|_{U^k}^{2^k/(k+1)}=\frac2p .
\]
In particular, for every prime $p\ge3$ Conjecture \ref{conj:main} fails for all sufficiently large $k$.
\end{theorem}

Section \ref{sec:mechanism} explains the mechanism behind both theorems. A coset $y+H\subseteq\Sigma$ contributes $|H|^{k+1}$ combinatorial $k$-cubes, and when these are essentially all the cubes, the quantity in \eqref{eq:conj} tends to $|E|\cdot c(\Sigma)/N^d$, where $c(\Sigma)$ is the size of the largest coset contained in $\Sigma$. In the limit $k\to\infty$, Conjecture \ref{conj:main} thus degenerates into $N^d\le|E|\cdot c(\Sigma)$,
which is \eqref{eq:DS} with $|\Sigma|$ replaced by the largest coset inside $\Sigma$. This fails whenever $\Sigma$ has a small number of cosets, most sharply for $\Sigma=\Z_p\setminus\{0\}$, where $c(\Sigma)=1$.

Finally, Section \ref{sec:survives} outlines what parts of the main conjecture are left unresolved. Conjecture \ref{conj:main} is a theorem for $k=2$, the cases $k=3,4,5$ remain open, and $k=3$ appears to be a special case because of an exact Fourier self-duality of the $U^3$ norm (Remark \ref{rem:selfdual}). We propose the restricted Conjecture \ref{conj:restricted} and the problem of determining the sharp exponent (Question \ref{q:exponent}).

\section{The counterexample on $\Z_3$}\label{sec:Z3}

Throughout this section and the next it is convenient to phrase the Gowers norm of an indicator function as a cube count.

\begin{notation}\label{not:cubes}
For $A\subseteq\Z_N^d$ and $k\ge1$ let
\[
T_k(A)\coloneq\#\Bigl\{(x,h_1,\dots,h_k)\in(\Z_N^d)^{k+1}\ :\ x+w\cdot h\in A\ \text{ for all } w\in\{0,1\}^k\Bigr\}.
\]
\end{notation}

Since $1_A$ is real and nonnegative, every factor in Definition \ref{def:gowers} is $0$ or $1$, and the product equals $1$ exactly when all $2^k$ points $x+w\cdot h$ lie in $A$. Hence
\begin{equation}\label{eq:cubecount}
\|1_A\|_{U^k}^{2^k}=\frac{T_k(A)}{N^{d(k+1)}},
\qquad\text{so}\qquad
\|1_A\|_{U^k}^{2^k/(k+1)}=\frac{T_k(A)^{1/(k+1)}}{N^{d}} .
\end{equation}
Consequently Conjecture \ref{conj:main} is equivalent to the purely combinatorial statement
\begin{equation}\label{eq:conjcubes}
T_k(\Sigma)\ \ge\ \left(\frac{N^d}{|E|}\right)^{k+1}.
\end{equation}
For $h=(h_1,\dots,h_k)\in(\Z_N^d)^k$ write
\[
S(h)\coloneq\Bigl\{\textstyle\sum_{i\in I}h_i\ :\ I\subseteq\{1,\dots,k\}\Bigr\}
=\{0,h_1\}+\{0,h_2\}+\cdots+\{0,h_k\}
\]
for its set of subset sums.

\begin{lemma}\label{lem:subsetsums}
Let $A\subseteq\Z_N^d$ and $k\ge1$. Then
\[
T_k(A)=\sum_{h\in(\Z_N^d)^k}\#\bigl\{x\in\Z_N^d\ :\ x+S(h)\subseteq A\bigr\}.
\]
In particular, if $A=\Z_N^d\setminus\{0\}$ then
\[
T_k(A)=\sum_{h\in(\Z_N^d)^k}\bigl(N^d-|S(h)|\bigr).
\]
\end{lemma}

\begin{proof}
As $w$ ranges over $\{0,1\}^k$, the quantity $w\cdot h$ ranges over $S(h)$ (with multiplicity), so the condition ``$x+w\cdot h\in A$ for all $w$'' is exactly ``$x+S(h)\subseteq A$''. Summing over $h$ gives the first identity. If $A=\Z_N^d\setminus\{0\}$, the condition becomes $-x\notin S(h)$, which has exactly $N^d-|S(h)|$ solutions $x$.
\end{proof}

\begin{lemma}\label{lem:Z3count}
Let $\Sigma=\{1,2\}=\Z_3\setminus\{0\}$. Then $T_k(\Sigma)=2(k+1)$ for every $k\ge1$.
\end{lemma}

\begin{proof}
We apply Lemma \ref{lem:subsetsums} and split according to the number $j$ of indices $i$ with $h_i\ne0$.

If $j=0$, that is $h=0$, then $S(h)=\{0\}$ and $|S(h)|=1$, contributing $3-1=2$.
If $j=1$, let $h_i\ne0$ and $h_{i'}=0$ for $i'\ne i$, then $S(h)=\{0,h_i\}$ and $|S(h)|=2$, contributing $3-2=1$ for each such $h$. There are $k$ choices of $i$ and $2$ choices of $h_i\in\{1,2\}$, hence $2k$ tuples, contributing $2k$ in total.
If $j\ge2$, choose $i\ne i'$ with $h_i,h_{i'}\ne0$. Then
\[
S(h)\supseteq\{0,h_i\}+\{0,h_{i'}\}=\{0,\ h_i,\ h_{i'},\ h_i+h_{i'}\}.
\]
If $h_i=h_{i'}$ this set is $\{0,h_i,2h_i\}=\Z_3$; if $h_i\ne h_{i'}$ then $\{h_i,h_{i'}\}=\{1,2\}$ and the set is again $\{0,1,2\}=\Z_3$. In either case $|S(h)|=3$ and the contribution is $0$.

Summing everything together, we get $T_k(\Sigma)=2+2k=2(k+1)$.
\end{proof}

\begin{proof}[Proof of Theorem \ref{thm:main}]
With $f=\delta_0-\delta_1$ we have $E=\{0,1\}$, and
\[
\hat f(m)=3^{-1/2}\bigl(1-\chi(-m)\bigr),\qquad \chi(t)=e^{2\pi i t/3},
\]
which vanishes precisely when $\chi(-m)=1$, i.e.\ when $m=0$. Hence $\Sigma=\{1,2\}$ and $|E|=2$.

By \eqref{eq:cubecount} and Lemma \ref{lem:Z3count},
\[
|E|\cdot\|1_\Sigma\|_{U^k}^{2^k/(k+1)}
=2\cdot\frac{T_k(\Sigma)^{1/(k+1)}}{3}
=\frac23\bigl(2(k+1)\bigr)^{1/(k+1)},
\]
which is \eqref{eq:mainvalue}. Writing $n=k+1$, this is $<1$ if and only if $(2n)^{1/n}<3/2$, i.e.\ if and only if
\[
\psi(n)\coloneq\Bigl(\tfrac32\Bigr)^{n}-2n>0 .
\]
We have $\psi(6)=\bigl(\tfrac32\bigr)^{6}-12=11.390625-12<0$ and $\psi(7)=17.0859\ldots-14>0$. Moreover $\psi'(n)=\bigl(\tfrac32\bigr)^n\log\tfrac32-2>0$ as soon as $\bigl(\tfrac32\bigr)^n>2/\log\tfrac32=4.93\ldots$, which holds for $n\ge4$. Hence $\psi$ is strictly increasing on $[4,\infty)$ and $\psi(n)>0$ for $n\ge7$, that is, for $k\ge6$. Finally $\bigl(2(k+1)\bigr)^{1/(k+1)}\to1$, so the quantity tends to $2/3$.
\end{proof}

The values are as follows.

\begin{center}
\begin{tabular}{lcccccccc}
\toprule
$k$ & $2$ & $3$ & $4$ & $5$ & $6$ & $7$ & $8$ & $10$\\
\midrule
$T_k(\Sigma)$ & $6$ & $8$ & $10$ & $12$ & $14$ & $16$ & $18$ & $22$\\
$|E|\cdot\|1_\Sigma\|_{U^k}^{2^k/(k+1)}$
 & $1.2114$ & $1.1212$ & $1.0566$ & $1.0087$ & $\mathbf{0.9719}$ & $\mathbf{0.9428}$ & $\mathbf{0.9191}$ & $\mathbf{0.8830}$\\
\bottomrule
\end{tabular}
\end{center}

\begin{remark}
The $k=2$ entry $\tfrac23\cdot6^{1/3}=1.2114\ldots$ is consistent with \eqref{eq:U2}: here $\Lambda_2(\Sigma)=6$, $|E|=2$, and $2\cdot 6^{1/3}=3.63\ldots\ge 3=N$. The failure at $k=6$ is therefore caused by the growth of the exponent $2^k/(k+1)$ relative to the growth of $T_k(\Sigma)$.
\end{remark}

\section{Counterexamples on $\Z_p$}\label{sec:mechanism}

\begin{proof}[Proof of Theorem \ref{thm:allp}]
As above, $f=\delta_0-\delta_1$ has $\hat f(m)=p^{-1/2}(1-\chi(-m))$, so $E=\{0,1\}$ and $\Sigma=\Z_p\setminus\{0\}$.

\emph{Upper bound for $T_k(\Sigma)$.} Fix $h\in\Z_p^k$ and let $j=\#\{i:h_i\ne0\}$. Discarding the zero coordinates and applying the Cauchy--Davenport inequality $|A+B|\ge\min(p,|A|+|B|-1)$ repeatedly to the $j$ two-element sets $\{0,h_i\}$ gives
\[
|S(h)|\ \ge\ \min\bigl(p,\ j+1\bigr).
\]
By Lemma \ref{lem:subsetsums} the contribution of $h$ to $T_k(\Sigma)$ is $p-|S(h)|\le\max(0,p-1-j)$, which vanishes for $j\ge p-1$. Grouping by $j$,
\[
T_k(\Sigma)\ \le\ \sum_{j=0}^{p-2}\binom kj (p-1)^j\,(p-1-j).
\]
For $0\le j\le p-2$ and $k\ge1$ we have $\binom kj\le k^{j}\le k^{p-2}$ and $(p-1)^j(p-1-j)\le p^{p-1}$, so
\begin{equation}\label{eq:polybound}
T_k(\Sigma)\ \le\ p^{p}\,k^{p-2}
\qquad (k\ge1).
\end{equation}
Thus $T_k(\Sigma)$ grows at most polynomially in $k$, with degree depending only on $p$.

\emph{Lower bound.} Taking $h=0$ and $x\in\Sigma$ gives $T_k(\Sigma)\ge p-1\ge2$. Hence by \eqref{eq:cubecount},
\[
|E|\cdot\|1_\Sigma\|_{U^k}^{2^k/(k+1)}=\frac{2}{p}\,T_k(\Sigma)^{1/(k+1)},
\]
and $2\le T_k(\Sigma)\le p^pk^{p-2}$ forces $T_k(\Sigma)^{1/(k+1)}\to1$ as $k\to\infty$ with $p$ fixed. Hence the quantity tends to $2/p<1$, and in particular is $<1$ for all large $k$.
\end{proof}

The thresholds computed from Lemma \ref{lem:subsetsums} are given below; the entries are $|E|\cdot\|1_\Sigma\|_{U^k}^{2^k/(k+1)}$ for $\Sigma=\Z_p\setminus\{0\}$ and $|E|=2$.

\begin{center}
\begin{tabular}{lcccccc c}
\toprule
 & $k=3$ & $k=5$ & $k=6$ & $k=8$ & $k=10$ & $k=15$ & least $k$ with value $<1$\\
\midrule
$p=3$  & $1.1212$ & $1.0087$ & $0.9719$ & $0.9191$ & $0.8830$ & $0.8279$ & $6$\\
$p=5$  & $1.3454$ & $1.1130$ & $1.0310$ & $0.9122$ & $0.8313$ & $0.7112$ & $7$\\
$p=7$  & $1.4922$ & $1.2101$ & $1.1013$ & $0.9403$ & $0.8304$ & $0.6695$ & $8$\\
$p=11$ & $1.6583$ & $1.3523$ & $1.2130$ & $0.9999$ & $0.8538$ & $0.6428$ & $8$\\
$p=13$ & $1.7073$ & $1.4077$ & $1.2585$ & $1.0268$ & $0.8678$ & $0.6398$ & $9$\\
\bottomrule
\end{tabular}
\end{center}

First, the least failing $k$ increases with $p$: larger $p$ gives a smaller limit $2/p$ but a larger polynomial degree $p-2$ in \eqref{eq:polybound}, hence slower convergence. Second, for each fixed $k$ the minimum over $p$ is attained at a $p$ that grows with $k$ (at $p=3$ for $k\le7$, at $p=5$ for $k=8,9$, at $p=7$ for $k=10,11$, at $p=11$ for $k=12,13$, at $p=13$ for $k=14,15$), and the minimal value continues to decrease. The failure is therefore not a bounded defect that could be absorbed into a constant.

We now isolate the general mechanism. Recall $c(A)=\max\{|H|:H\le\Z_N^d,\ y+H\subseteq A\text{ for some }y\}$, the size of the largest coset contained in $A$; note $c(A)\ge1$ for $A\ne\emptyset$.

\begin{proposition}\label{prop:mechanism}
Let $A\subseteq\Z_N^d$ be nonempty. Then $T_k(A)\ge c(A)^{k+1}$ for every $k\ge1$, and consequently
\[
\liminf_{k\to\infty}\ |E|\cdot\|1_A\|_{U^k}^{2^k/(k+1)}\ \ge\ \frac{|E|\cdot c(A)}{N^d}.
\]
If in addition $T_k(A)\le P(k)\,c(A)^{k+1}$ for some polynomial $P$, then the limit exists and equals $|E|\,c(A)/N^d$.
\end{proposition}

\begin{proof}
Let $y+H\subseteq A$ with $|H|=c(A)$. For every $x\in y+H$ and every $h_1,\dots,h_k\in H$ we have $x+w\cdot h\in y+H\subseteq A$ for all $w\in\{0,1\}^k$, so these $|H|^{k+1}$ tuples are counted by $T_k(A)$. The two displayed consequences follow from \eqref{eq:cubecount} together with $P(k)^{1/(k+1)}\to1$.
\end{proof}

\begin{remark}[The limiting form of the conjecture]\label{rem:limitform}
Proposition \ref{prop:mechanism} shows that, for supports $\Sigma$ whose cube count is polynomially close to the trivial coset contribution, Conjecture \ref{conj:main} in the limit $k\to\infty$ asserts
\[
N^d\ \le\ |E|\cdot c(\Sigma),
\]
that is, the Donoho--Stark inequality \eqref{eq:DS} with $|\Sigma|$ replaced by the largest coset contained in $\Sigma$. Since $c(\Sigma)\le|\Sigma|$, with equality precisely when $\Sigma$ is a coset, this is a strictly stronger statement than \eqref{eq:DS} for every non-coset $\Sigma$, and it is false in general: the family of Theorem \ref{thm:allp} has $c(\Sigma)=1$ (the only subgroups of $\Z_p$ are trivial, and $\Sigma\ne\Z_p$), $|E|=2$, and $N^d=p$. Conjecture \ref{conj:main} is thus asymptotically equivalent to a coset-restricted uncertainty principle, which has no signal-theoretic justification.
\end{remark}

\begin{remark}[The exponent]\label{rem:exponent}
Suppose one seeks the largest exponent $e_k>0$ for which
\begin{equation}\label{eq:generalexp}
1\le |E|\cdot\|1_\Sigma\|_{U^k}^{e_k}
\end{equation}
holds for all nonzero $f:\Z_N^d\to\C$ and all $N,d$. Taking $f=1_{y+H}$ so that $\Sigma=H^\perp$ and $|E|=N^d/|H^\perp|$, one gets from \eqref{eq:cubecount} that \eqref{eq:generalexp} reads
$\bigl(N^d/|H^\perp|\bigr)^{1-e_k(k+1)/2^k}\ge1$, which forces
\[
e_k\ \le\ \frac{2^k}{k+1}.
\]
Thus, the exponent in Conjecture \ref{conj:main} is the largest one the coset examples permit. However, Theorem \ref{thm:main} shows it is nevertheless too large. Applying \eqref{eq:generalexp} instead to the family of Theorem \ref{thm:allp} and using \eqref{eq:polybound} gives, for each fixed prime $p\ge3$,
\[
e_k\ \le\ \frac{2^k\log 2}{(k+1)\log p-\log T_k(\Z_p\setminus\{0\})}\ =\ \bigl(1+o_p(1)\bigr)\,\frac{2^k\log_p 2}{k+1}
\qquad (k\to\infty).
\]
The constraint from a fixed $p$ is vacuous for $k$ small relative to $p$, and the table above indicates that the binding prime grows with $k$, which suggests a sharp exponent of order $2^k/(k\log k)$. Determining it is Question \ref{q:exponent} below.
\end{remark}

\section{What remains}\label{sec:survives}

For $k=2$, Conjecture \ref{conj:main} is precisely the additive energy uncertainty principle \eqref{eq:additive}, hence a theorem \cite{2025additiveenergyuncertaintyprinciple}; Theorems \ref{thm:main} and \ref{thm:allp} dispose of $k\ge6$. The cases $k=3,4,5$ remain open, and the families used here do not decide them: by \eqref{eq:mainvalue} the signal $f=\delta_0-\delta_1$ on $\Z_3$ gives the values $1.1212\ldots$, $1.0566\ldots$, $1.0087\ldots$ at $k=3,4,5$, and by the table of Section \ref{sec:mechanism} the corresponding signal on $\Z_p$ gives strictly larger values for every prime $p\ge5$.

Let us record the following reformulation, which removes the signal entirely. By \eqref{eq:cubecount} the quantity in \eqref{eq:conj} depends on $f$ only through $|E|$ and $\Sigma$, and $T_k$ is monotone under inclusion of sets. Hence, when setting
\[
m(\Sigma)\coloneq\min\bigl\{|\supp(g)|\ :\ g\ne0,\ \supp(\hat g)\subseteq\Sigma\bigr\},
\]
the minimum weight of the space of functions with Fourier support in $\Sigma$, Conjecture \ref{conj:main} is equivalent to \eqref{eq:conjcubes} with $|E|$ replaced by $m(\Sigma)$, namely
\begin{equation}\label{eq:reform}
T_k(\Sigma)\ \ge\ \left(\frac{N^d}{m(\Sigma)}\right)^{k+1}
\qquad\text{for every nonempty }\Sigma\subseteq\Z_N^d .
\end{equation}
Proposition \ref{prop:mechanism} gives the lower bound $T_k(\Sigma)\ge c(\Sigma)^{k+1}$, and the Donoho--Stark inequality \eqref{eq:DS} gives $N^d/m(\Sigma)\le|\Sigma|$. Thus \eqref{eq:reform} already follows from the coset contribution alone whenever $c(\Sigma)\ge N^d/m(\Sigma)$, and in particular whenever $\Sigma$ is a coset, where both quantities equal $|\Sigma|$. 

\begin{conjecture}\label{conj:restricted}
Conjecture \ref{conj:main} holds for $2\le k\le 5$, with equality if and only if $\Sigma$ is a coset of a subgroup of $\Z_N^d$.
\end{conjecture}

\begin{question}\label{q:exponent}
Determine the largest sequence $(e_k)_{k\ge2}$ for which $1\le|E|\cdot\|1_\Sigma\|_{U^k}^{e_k}$ holds for all nonzero $f:\Z_N^d\to\C$ and all $N,d$. By Remark \ref{rem:exponent}, $e_k\le 2^k/(k+1)$, with equality for $k=2$ and, conjecturally, for $3\le k\le5$, and $e_k=O\bigl(2^k/(k\log k)\bigr)$ is suggested by the family of Theorem \ref{thm:allp}.
\end{question}

The case $k=3$ is special among the remaining ones for a structural reason, which we write down in case it is useful.

\begin{remark}[Self-duality at $k=3$]\label{rem:selfdual}
Let $H_k=\{(x+w\cdot h)_{w\in\{0,1\}^k}\ :\ x,h_1,\dots,h_k\in\Z_N^d\}\le(\Z_N^d)^{2^k}$,
the order-$N^{d(k+1)}$ subgroup used to define $\|f\|_{U^k}^{2^k}$. For each vertex $w\in\{0,1\}^k$, Fourier inversion in the normalization above gives
\[
\Conj^{|w|} g(x+w\cdot h)
= N^{-d/2}\sum_{\xi_w\in\Z_N^d}\Conj^{|w|}\hat g(\xi_w)\,
\chi\!\bigl((-1)^{|w|}\xi_w\cdot(x+w\cdot h)\bigr).
\]
The sign $(-1)^{|w|}$ arises because an odd number of conjugations replaces $\chi$ by
$\overline\chi$. Write $\epsilon_w=(-1)^{|w|}$. By multiplying over the $2^k$ vertices, we get that
\[
\prod_{w\in\{0,1\}^k}\Conj^{|w|}g(x+w\cdot h)
= N^{-d2^{k-1}}\!\!\sum_{(\xi_w)_w\in(\Z_N^d)^{2^k}}
\Bigl(\prod_w\Conj^{|w|}\hat g(\xi_w)\Bigr)\,
\chi\!\Bigl(\sum_w\epsilon_w\,\xi_w\cdot(x+w\cdot h)\Bigr).
\]
Since $w\cdot h=\sum_i w_ih_i$, the phase is linear in the $k+1$ free parameters
$x,h_1,\dots,h_k$:
\[
\sum_w\epsilon_w\,\xi_w\cdot(x+w\cdot h)
=\Bigl(\sum_w\epsilon_w\xi_w\Bigr)\cdot x
+\sum_{i=1}^k\Bigl(\sum_{w:\,w_i=1}\epsilon_w\xi_w\Bigr)\cdot h_i .
\]
Averaging over $x,h_1,\dots,h_k\in\Z_N^d$ and using
$N^{-d}\sum_{y\in\Z_N^d}\chi(a\cdot y)=\mathbf 1_{\{a=0\}}$ once for each parameter absorbs the
factor $N^{-d(k+1)}$ of Definition \ref{def:gowers} and forces the $k+1$ linear conditions
\[
\sum_{w}\epsilon_w\xi_w=0,
\qquad
\sum_{w:\,w_i=1}\epsilon_w\xi_w=0\quad(1\le i\le k),
\]
which define the subgroup $L_k\le(\Z_N^d)^{2^k}$. Setting $\eta_w=\epsilon_w\xi_w$, these
conditions say that $\eta$ lies in the annihilator $H_k^{\perp}$ of $H_k$ under the
pairing $\langle\eta,y\rangle=\sum_w\eta_w\cdot y_w$. Since $(\xi_w)_w\mapsto(\epsilon_w\xi_w)_w$
is an automorphism of $(\Z_N^d)^{2^k}$, it carries $L_k$ onto $H_k^{\perp}$, so
\[
|L_k|=|H_k^{\perp}|=\frac{N^{d2^k}}{|H_k|}=N^{d(2^k-(k+1))}.
\]
Hence for every $g:\Z_N^d\to\C$,
\[
\|g\|_{U^k}^{2^k}=N^{-d2^{k-1}}\sum_{\xi\in L_k}\ \prod_{w\in\{0,1\}^k}\Conj^{|w|}\hat g(\xi_w).
\]
Since $|H_k|=N^{d(k+1)}$ and $|L_k|=N^{d(2^k-k-1)}$, these subgroups of $(\Z_N^d)^{2^k}$ can
coincide only if $k+1=2^k-(k+1)$, i.e.\ $2^k=2(k+1)$, which holds among positive integers
only at $k=3$.

At $k=2$, $L_2$ collapses to the diagonal $\{\xi_{00}=\xi_{01}=\xi_{10}=\xi_{11}\}$, and
applying the identity to $g=\hat f$ (using $\widehat{\hat f}(\xi)=f(-\xi)$) gives
\[
\|\hat f\|_{U^2}^4=N^{-2d}\sum_{x\in\Z_N^d}|f(x)|^4
\qquad\text{for all } f:\Z_N^d\to\C.
\]
This relates the $U^2$ norm of $\hat f$ to the $\ell^4$ norm of $f$. The ratio $\|\hat f\|_{U^2}^4/\|f\|_{U^2}^4$ is not
constant in $f$.

At $k=3$, however, $H_3=L_3$, and $H_3$ is symmetric under negation (it is a
subgroup), so reparametrizing $\xi\mapsto-\xi$ in the sum above removes the reflection coming
from $\widehat{\hat f}(\xi)=f(-\xi)$ entirely. This gives the identity
\[
\|\hat f\|_{U^3}^8=\|f\|_{U^3}^8
\qquad\text{for all } f:\Z_N^d\to\C,
\]
with \emph{no} $N$-dependent constant. Thus Conjecture \ref{conj:main} at $k=3$ is invariant under $f\mapsto\hat f$. No such identity holds for $k\ne3$.
\end{remark}

\section*{Acknowledgements}
The author thanks the coauthors of \cite{RefinedUncertainty} for formulating the conjecture.

\bibliographystyle{alpha}
\bibliography{bibtex}

@article{DonohoStark,
  author  = {Donoho, David L. and Stark, Philip B.},
  title   = {Uncertainty Principles and Signal Recovery},
  journal = {SIAM Journal on Applied Mathematics},
  volume  = {49},
  number  = {3},
  pages   = {906--931},
  year    = {1989},
  doi     = {10.1137/0149053}
}

@article{MatolcsiSzucs,
  author  = {Matolcsi, Tam{\'a}s and Sz{\H u}cs, J{\'o}zsef},
  title   = {Intersection des mesures spectrales conjugu\'ees},
  journal = {Comptes Rendus de l'Acad\'emie des Sciences Paris S\'erie A-B},
  volume  = {277},
  pages   = {A841--A843},
  year    = {1973}
}

@misc{2025additiveenergyuncertaintyprinciple,
  author        = {Aldahleh, Karam and Iosevich, Alex and Iosevich, J. and
                    Jaimangal, J. and Mayeli, Azita and Pack, S.},
  title         = {Additive Energy, Uncertainty Principle and Signal Recovery Mechanisms},
  year          = {2025},
  eprint        = {2504.14702},
  archivePrefix = {arXiv},
  primaryClass  = {math.CA}
}

@misc{RefinedUncertainty,
  author        = {Bortnovskyi, Ivan and Duvivier, June and Iosevich, Alex and
                    Iosevich, Josh and Kwon, Say-Yeon and Laurence, Meiling and
                    Lucas, Michael and Pan, Tiancheng and Palsson, Eyvindur and
                    Smucker, Jennifer and Vranesko, Iana},
  title         = {Refined Additive Uncertainty Principle},
  year          = {2025},
  note          = {Accepted for publication, Pacific Journal of Mathematics},
  eprint        = {2510.26664},
  archivePrefix = {arXiv},
  primaryClass  = {math.CA}
}

@book{Tao-Vu-2006,
  author    = {Tao, Terence and Vu, Van},
  title     = {Additive Combinatorics},
  series    = {Cambridge Studies in Advanced Mathematics},
  volume    = {105},
  publisher = {Cambridge University Press},
  address   = {Cambridge},
  year      = {2006}
}

\end{document}